\documentclass[10pt]{amsart}

\usepackage{geometry} 
\usepackage{graphicx,url}
\usepackage{amssymb,amsmath, amsthm}
\usepackage{epstopdf}
\usepackage{subfig}
\usepackage[applemac]{inputenc}
\newcommand{\de}{\mathrm{d}} 
\newcommand{\se}{\mathbb{S}}
\newcommand{\re}{\mathbb{R}}

\usepackage{graphicx,url}
\newtheorem{theorem}{Theorem}
\newtheorem{lemma}{Lemma}
\newtheorem{remark}{Remark}[section]

\newtheorem{definition}{Definition}

\title{Reflecting random flights}
\numberwithin{equation}{section}
\author{ Alessandro De Gregorio}

\author{Enzo Orsingher}

\address{Dipartimento di Scienze Statistiche, ``Sapienza" University of Rome,
P.le Aldo Moro, 5 - 00185, Rome, Italy}
\email{alessandro.degregorio@uniroma1.it}
\email{enzo.orsingher@uniroma1.it}
\allowdisplaybreaks

\date{\today}
\keywords{Bessel functions, circular inversion, Euler-Darboux-Poisson equations, fractional-type Poisson distributions, random motions at finite velocity, inversion in hyperplane.}
\begin{document}

\maketitle

\begin{abstract}

We consider random flights in $\re^d$ reflecting on the surface of a sphere $\se^{d-1}_R,$ with center at the origin and with radius $R,$ where reflection is performed by means of circular inversion. Random flights studied in this paper are motions where the orientation of the deviations are uniformly distributed on the unit-radius sphere $\se^{d-1}_1$.

We obtain the explicit probability distributions of the position of the moving particle when the number of changes of direction is fixed and equal to $n\geq 1$. We show that these distributions involve functions which are solutions of the Euler-Darboux-Poisson equation. The unconditional probability distributions of the reflecting random flights are obtained by suitably randomizing $n$ by means of a fractional-type Poisson process.

Random flights reflecting on hyperplanes according to the optical reflection form are considered and the related distributional properties derived.
\end{abstract}


\section{Introduction}

This paper is concerned with random flights, that is with random motions in the Euclidean space $\re^d,d\geq 2$, performed at finite velocity (sometimes also called random evolutions). These continuous-time non-Markovian random motions have sample paths formed by straight lines, turning through any angle whatever, i.e. with uniformly distributed directions on the unit-radius sphere.

These random flights are useful to describe concrete motions of particles in gases and for these reason have been investigated by many physicists and mathematicians over the years. The first ones who analyzed the random flights with a fixed number of deviations were K. Pearson and J.C. Kluyver and some years later, S. Chandrasekar wrote a long paper on this subject with applications to astronomy.

Recent applications to the analysis of photon propagation in the Cosmic Microwave Background (CMB) radiation have been discussed in Reimberg and Abramo (2013). Furthermore, Martens {\it et al}. (2012) have shown that the probability law of planar random motions coincides with the explicit form of the van Hove function for the run-and-tumble model in two dimensions. This work gives an interesting and strong link between explicit solutions of the Lorentz model of electron conduction and the probability theory of random flights. For other possible applications of the random flights models, see Hughes (1995).

The main object of our investigation is represented by the position $\{\underline{\bf X}_d(t),t>0\}$ reached after a fixed or a random number of deviations. The randomization of the number $\mathcal N(t)$ of steps and of the lengths of intermediate displacements has produced more flexible versions of random flights for which explicit distributions of  $\{\underline{\bf X}_d(t),t>0\}$ have been obtained.

Recently many papers have studied random motions with velocity $c>0,$ with uniformly distributed deviations at Poisson paced times in the plane and successively in the Euclidean space $\re ^d$ (see, for instance, Stadje, 1987, Masoliver {\it et al}., 1993, Kolesnik and Orsingher, 2005, for planar motions, Stadje, 1989, and Orsingher and De Gregorio, 2007, in $\re^d$). Franceschetti, (2007) established a relationship between the number of deviations $n$ and the dimension $d$ for which the conditional distributions $P\{\underline{\bf X}_d(t)\in \de \underline{\bf x}_d|\mathcal N(t)=n\}$ are uniform. The assumption that changes of direction are governed by a homogeneous Poisson process leads to explicit (conditional and unconditional) probability distributions of $\{\underline{\bf X}_d(t),t>0\}$ only for $d=2,4.$ The idea of assuming different forms for intermediate steps (displacements) has produced fruitful results in that the probability distribution of $\underline{\bf X}_d(t)$ can be explicitly produced for all spaces of dimension $d\geq 2.$

The random flight $(\underline\Theta,\underline\tau,\mathcal N(t)),$ is a triple where $\underline\Theta$ represents the ensemble of deviations during the time interval $[0,t]$, $\underline\tau$ is the vector of the lengths of the displacements for a fixed number $\mathcal N(t)$ number of changes of direction. These processes have been  extensively investigated in the case where $\underline\tau$ has a Dirichlet distribution and $\underline\Theta$ is spherically uniform. The papers by  Le Ca\"er (2010), (2011), De Gregorio and Orsingher (2012), De Gregorio (2014) and Letac and Piccioni (2014) are devoted to this case. Pogorui and Rodriguez-Dagnino (2011), (2013) have studied random flights where $\underline\tau$ has Erlang distribution, while Beghin and Orsingher (2010) obtained conditional distributions of $\{\underline{\bf X}_d(t),t>0\}$ where the steps $\underline\tau$ are Gamma random variables with parameter 2. 

The equations governing the unconditional probability distributions of the random flight $(\underline\Theta,\underline\tau,\mathcal N(t))$ have recently been obtained (see Garra and Orsingher, 2014). The case where $\underline\Theta$ has a specific non-uniform distribution has been studied by De Gregorio (2012). Motions on subspaces of $\re^d$ can be regarded as random flights with random velocities and are studied in De Gregorio and Orsingher (2012) and Pogorui and Rodriguez-Dagnino (2012). Asymptotic results for the position of these type of random walks have been obtained by Ghosh {\it et al}. (2014).

We now describe in detail the structure of the random flights. Let us consider a particle or a walker which starts from the origin of $\mathbb{R}^d,d\geq 2$, and performs its motion with a constant velocity $c>0$. We indicate by $0=t_0<t_1<t_2<...<t_k<...$ the random instants at which the random walker changes direction and denote the length of time separating these instants by $\tau_k=t_k-t_{k-1}$, $k\geq 1$. Let $\mathcal N(t)=\sup\{k\geq 1:t_k\leq t\}$ be the (random) number of times in which the random motion changes direction during the interval $[0,t]$. If, at time $t>0$, $ \mathcal N(t)=n$, with $n\geq 1$, the random motion has performed $n+1$ displacements. We observe that ${\bf \underline\tau}_n=(\tau_1,...,\tau_n)\in S_n$, where $S_n$ represents the open simplex
$$ S_n=\left\{(\tau_1,...,\tau_n)\in\mathbb R^n: 0<\tau_k<t-\sum_{j=0}^{k-1}\tau_j, k=1,2,...,n\right\},$$ with $\tau_0=0$ and $\tau_{n+1}=t-\sum_{j=1}^n\tau_j$.

By $\underline\Theta_{n+1}=(\underline\theta_{d-1}^1,...,\underline\theta_{d-1}^k,...,\underline\theta_{d-1}^{n+1})$ we denote the vector of independent deviations performed at times $t_k,k=0,1,...,n.$ The $k$-th random variable $\underline\theta_{d-1}^k$ is a $(d-1)$-dimensional random variable with components $(\theta_{1}^k,...,\theta_{d-2}^k,\phi^k)$ distributed uniformly on the unit-radius sphere $\mathbb{S}_{1 }^{d-1}=\{\underline{\bf x}_d\in \mathbb R^d:||\underline{\bf x}_d||=1\}$. This means that the probability density function of $\underline{\theta}_{d-1}^k$ is equal to
\begin{equation}\label{eq:jointdis1}
\varphi(\underline{\theta}_{d-1}^k)=\frac{\Gamma(\frac d2)}{2\pi^{\frac d2}}\sin^{d-2}\theta_1^k\sin^{d-3}\theta_2^k\cdots \sin\theta_{d-2}^k,
\end{equation}
where $\theta_j^k\in[0,\pi],j\in\{1,...,d-2\},\,\phi^k\in[0,2\pi]$. Furthermore, $\tau_k$ and $\theta_j^k$ are independent for each $k$.

Let us denote by $\{\underline{\bf X}_d(t),t>0\}$ the process representing the position reached, at time $t>0$, by the particle moving randomly according to the rules described above. The position $\underline{\bf X}_d(t)=(X_1(t),...,X_d(t)),$ is the main object of interest of the random flight and can be written, for $\mathcal N(t)=n,$ as
\begin{equation}\label{eq:definition}
\underline{\bf X}_d(t)=c\sum_{k=1}^{n+1}{\bf \underline{V}}_d^k\tau_k,
\end{equation}
where ${\bf\underline{V}}_{d}^k,k=1,2,...,n+1,$ are independent $d$-dimensional random vectors defined as follows
$${\bf\underline{V}}_d^k=\left(
 \begin{array}{l}

               \sin\theta_1^k\sin\theta_2^k\cdot\cdot\cdot\sin\theta_{d-2}^k\sin\phi^{k} \\
       \sin\theta_1^k\sin\theta_2^k\cdot\cdot\cdot\sin\theta_{d-2}^k\cos\phi^{k} \\
      ...\\
      \sin\theta_1^k\cos\theta_2^k\\
       \cos\theta_1^k
   \end{array}    
    \right)$$
    and $(\theta_1^k,\theta_2^k,...,\theta_{d-2}^k,\phi^{k})$ are independent and identically distributed with density \eqref{eq:jointdis1}.

In this work we analyze reflecting random walks defined by means of inversive geometry. In particular,
this paper studies random flights reflecting on the $d$-dimensional sphere $\se_R^{d-1}=\{\underline{\bf x}_d:||\underline{\bf x}_d||=R\}$ with radius $R.$ The reflected processes are constructed by suitably manipulating the sample paths of the free random flight $\{\underline{\bf X}_d(t),t>0\}$. We assume that the part of the trajectories of the free random flight are reflected by inversion with respect to the sphere $\se_R^{d-1}$. This produces substantial changes in their form, because the segments of outside lying sample paths are converted into arcs of circle inside $\se_R^{d-1}$. The picture of the reflecting processes is therefore made up by straight lines (for sample paths which never crossed  $\se_R^{d-1}$) and circular arcs (for the sample paths which performed excursions outside $\se_R^{d-1}$). The same approach is used by Aryasova {\it et al}. (2013) in the study of reflecting diffusion processes.

The circular inversion permits us to write down the probability distribution $p_n^*(\underline{{\bf x}}_d,t)$ of the reflected process $\{\underline{{\bf X}}_d^*(t),t>0\}$ by exploiting the density function $p_n(\underline{{\bf x}}_d,t)$ of the free random flight $\{\underline{{\bf X}}_d(t),t>0\}.$ We show that
\begin{align}\label{eq:condensintrod}
p_{n}^*(\underline{{\bf x}}_d,t)&=\begin{cases}
p_n(\underline{{\bf x}}_d,t){\bf 1}_{B_R^d}(\underline{{\bf x}}_d),& t\leq\frac{R}{c}\\
p_n(\underline{{\bf x}}_d,t){\bf 1}_{B_R^d}(\underline{{\bf x}}_d)+\frac{R^{2d}}{||\underline{{\bf x}}_d||^{2d}}p_n\left(R^2\frac{\underline{{\bf x}}_d}{||\underline{{\bf x}}_d||^2},t\right){\bf 1}_{C^d_{\frac{R^2}{ct},R}}(\underline{{\bf x}}_d),& t>\frac{R}{c},
\end{cases}
\end{align}
where $B_R^d$ is the ball in the space $\re^d$ with radius $R$, $C^d_{\frac{R^2}{ct},R}=\{\underline{{\bf x}}_d\in\mathbb R^d:\frac{R^2}{ct}<  ||\underline{{\bf x}}_d||\leq R\}$ and ${\bf 1}_A(x)$ is the indicator function for the set $A$.  
For $t>\frac Rc,$ formula \eqref{eq:condensintrod} registers the contribution of excursions outside $\se_R^{d-1}$ which are brought inside the sphere $\se_R^{d-1}$ by circular inversion. If $p_n(\underline{{\bf x}}_d,t)$ refers to random flights with Dirichlet distributed displacements we have that (see formulas (2.10) and (2.11) in De Gregorio and Orsingher, 2012)
\begin{equation}\label{eq:condensdir}
p_n(\underline{{\bf x}}_d,t)=\frac{\Gamma(\frac{n+1}{2}(d-h)+\frac h2)}{\Gamma(\frac n2(d-h))}(c^2t^2-||\underline{{\bf x}}_d||^2)^{\frac n2(d-h)-1},\quad ||\underline{{\bf x}}_d||<ct,h=1,2.
\end{equation}

The space-dependent factor appearing in \eqref{eq:condensdir} is a function satisfying the Euler-Poisson-Darboux (EPD) equation. In other words, the function
$$f_\beta(\underline{{\bf x}}_d,t):=(c^2t^2-||\underline{{\bf x}}_d||^2)^\beta,\quad \beta\in\re,||\underline{{\bf x}}_d||<ct,$$
satisfies the EPD equation
\begin{equation}
\frac{\partial ^2u}{\partial t^2}=c^2\Delta u+\frac{2\beta-1+d}{t}\frac{\partial u}{\partial t},
\end{equation}
where $\Delta:=\sum_{k=1}^d\frac{\partial^2}{\partial x_k^2}$, which becomes the $d$-dimensional wave equation for $\beta=\frac{d-1}{2}$. In the reflected motions considered here we must examine the functions
$$\bar f_\beta(\underline{{\bf x}}_d,t):=\left(c^2t^2-\frac{R^4}{||\underline{{\bf x}}_d||^2}\right)^\beta,\quad \beta\in\re,\frac{R^2}{ct}<||\underline{{\bf x}}_d||<R,$$
which also are solutions of more complicated Euler-Poisson-Darboux equations where space-varying coefficients appear.

The final section of the paper is concerned with random flights reflecting when colliding with hyperplanes. In this case reflection is intended in the sense that striking and reflecting paths form the same angle with respect to the normal to the hyperplane $H(\underline a_d,b)=\{\underline{{\bf x}}_d\in\re ^d:\, \langle\underline a_d,\underline{{\bf x}}_d\rangle=b;\,\underline a_d\in\re^d,b\in\re\}$. In this sense reflecting random flights behave as light rays of optics when the reflecting surface is a hyperplane. The trajectories of the reflecting random motions can be obtained form those of the free random flights by means of the bijective operator $\nu:\re^d\to \re^d$ defined as
 \begin{equation}
\nu(\underline{{\bf x}}_d):=\underline{{\bf x}}_d+2\frac{ b-\langle\underline a_d,\underline{{\bf x}}_d\rangle}{\langle\underline a_d,\underline a_d\rangle}\underline a_d,
\end{equation}
which represents the reflection on the hyperplane $H(\underline a_d,b)$. For $t>t':=\inf(t:H(\underline a_d,b)\cap  B_{ct}^d\neq \varnothing ),$ we therefore have that the density $p_n'(\underline{{\bf x}}_d,t)$ of the reflected random motion $\{\underline{\bf X}_d'(t),t>0\}$ reads

\begin{align}
p_{n}'(\underline{{\bf x}}_d,t)
&=\begin{cases}
p_n(\underline{{\bf x}}_d,t){\bf 1}_{B_{ct}^d}(\underline{{\bf x}}_d),& t\leq t'\\
p_n(\underline{{\bf x}}_d,t){\bf 1}_{L_{ct}^d}(\underline{{\bf x}}_d)+p_n\left(\nu(\underline{{\bf x}}_d),t\right){\bf 1}_{V_{ct}^d}(\underline{{\bf x}}_d),& t>t',
\end{cases}
\end{align}
where $L_{ct}^d:=L_{ct}^d(\underline a_d, b):=\{\underline{{\bf x}}_d\in \mathbb{R}^d:||\underline{{\bf x}}_d||^2< c^2t^2,\langle\underline a_d,\underline{{\bf x}}_d\rangle<b\}$ and $V_{ct}^d$ consists of points of $B_{ct}^d$ reflected by the operator $\nu$ (see Figure \ref{fig2}).

The main concepts of the inversive geometry represent fundamental tools for our analysis. Therefore, for the convenience of the reader, we recall some basic definitions in the Appendices \ref{refspheres}-\ref{refhyper}.
\section{Notations}

We list the main symbols used in this paper.
\begin{itemize}
\item With $||\cdot||$ and $\langle\cdot,\cdot\rangle$ we indicate the Euclidean norm and the scalar product, respectively.
\item  Let $\se_R^{d-1}(\underline{{\bf x}}_d^0):=\{\underline{{\bf x}}_d\in\mathbb R^d:||\underline{{\bf x}}_d-\underline{{\bf x}}_d^0||= R\}$ be the sphere with radius $R>0$ and center at $\underline{{\bf x}}_d^0\in\re ^d$. We set $\se_R^{d-1}(O):=\se_R^{d-1}$ where $O$ is the origin of $\re^d$ . Furthermore, the surface area of $\mathbb{S}_{1 }^{d-1}$ is given by
$\text{area}(\mathbb{S}_{1 }^{d-1})=\frac{2\pi^{\frac d2}}{\Gamma(\frac d2)}.$ 
\item Let $B_R^d:=\{\underline{{\bf x}}_d\in\mathbb R^d:||\underline{{\bf x}}_d||< R\}$ be the ball with center $O$ and radius $R$.  Let $C_{R_1,R_2}^d:=\{\underline{{\bf x}}_d\in\mathbb R^d:R_1< ||\underline{{\bf x}}_d||\leq R_2\}$.

\item Let $H(\underline a_d,b):=\{\underline{{\bf x}}_d\in\re ^d:\, \langle\underline a_d,\underline{{\bf x}}_d\rangle=b;\,\underline a_d\in\re^d,b\in\re\}$ be a hyperplane in $\re^d$.
\item Let ${\bf 1}_A(x)$ be the indicator function, that is  
$${\bf 1}_A(x)=\begin{cases}1,& x\in A,\\
0,& 	\text{otherwise},
\end{cases}$$
while $$J_\nu(x)=\sum_{k=0}^\infty(-1)^k\left(\frac{x}{2}\right)^{2k+\nu}\frac{1}{k!\Gamma(k+\nu+1)},\quad x,\nu\in\mathbb{R},$$
is the Bessel function of order $\nu$.
\item Let $n$ be the fixed number of changes of direction in the time interval $[0,t]$, we indicate the conditional probability areaure by $$P_n\{\cdot\in A\}:=P\{\cdot\in A|\mathcal N(t)=n\},$$
for all Borel sets $A$, with $n\geq 1$. Furthermore, let $E_n\{\cdot\}:=E\{\cdot|\mathcal N(t)=n\}$.

\item Let 
$$p_n(\underline{{\bf x}}_d,t):=P_n\{\underline{{\bf X}}_d(t)\in \de\underline{{\bf x}}_d\}/\prod_{k=1}^d\de x_k,$$
$$p_n^*(\underline{{\bf x}}_d,t):=P_n\{\underline{{\bf X}}_d^*(t)\in \de\underline{{\bf x}}_d\}/\prod_{k=1}^d\de x_k,$$
$$p_n'(\underline{{\bf x}}_d,t):=P_n\{\underline{{\bf X}}_d'(t)\in \de\underline{{\bf x}}_d\}/\prod_{k=1}^d\de x_k.$$
\end{itemize}

\section{Preliminary results on random flights}\label{sec:prelim}

Let us indicate by $g({\bf \underline\tau}_n; t)$ the probability density function of the random vector ${\bf \underline\tau}_n.$
We provide the probability distribution of $\{\underline{{\bf X}}_d(t),t>0\}$ when the number of steps performed by the motion in $[0,t]$ is fixed.

\begin{lemma} \label{lemma1}
Let $n\geq 1$ be the number of changes of direction happening during time interval $[0,t]$. The conditional density function of $\{\underline{{\bf X}}_d(t),t>0\}$ is equal to
\begin{align}\label{eq:densfree}
p_n(\underline{{\bf x}}_d,t)=\frac{\left\{2^{\frac d2-1}\Gamma\left(\frac d2\right)\right\}^{n+1}}{(2\pi)^{\frac d2}||\underline{\bf x}_d||^{\frac d2-1}}\int_0^{\infty}\left[\int_{S_n}g({\bf \underline\tau}_n; t)\prod_{k=1}^{n+1}\left\{\frac{J_{\frac d2-1}(c\tau_k\rho)}{(c\tau_k\rho)^{\frac d2-1}} \right\}\prod_{k=1}^{n}\de\tau_k\right] \rho^{\frac d2}J_{\frac d2-1}(\rho ||\underline{\bf x}_d||)
\de\rho,
\end{align}
with $ ||\underline{\bf x}_d||<ct$.
\end{lemma}
\begin{proof}
The expression \eqref{eq:densfree} has been obtained by Orsingher and De Gregorio (2007) (formula (2.13)) in the uniform case and by De Gregorio (2014) (formula (3.1)) in the case of generalized Dirichlet distributions. Similar steps can be used in a general framework for the distribution $g({\bf \underline\tau}_n; t)$. In what follows, we will provide a sketch of the proof.

Let us start the proof by showing that the characteristic function of $\underline{\bf X}_d(t)$ is equal to 
\begin{align}\label{eq:cf}
\mathcal F_n(\underline{\bf \alpha}_d)
&:=E_n\left\{e^{i\langle\underline{\bf \alpha}_d,\underline{{\bf X}}_d(t)\rangle}\right\}\notag\\
&=\left\{2^{\frac d2-1}\Gamma\left(\frac d2\right)\right\}^{n+1}\int_{S_n}g({\bf \underline\tau}_n; t)\prod_{k=1}^{n+1}\left\{\frac{J_{\frac d2-1}(c\tau_k||\underline{\alpha}_d||)}{(c\tau_k||\underline{\alpha}_d||)^{\frac d2-1}} \right\}\prod_{k=1}^{n}\de\tau_k.
\end{align}
 We can write that
\begin{align*}
\mathcal F_n(\underline{\bf \alpha}_d)&=\int_{S_n}g({\bf \underline\tau}_n; t)\mathcal{I}_n(\underline{\alpha}_d;{\bf \underline\tau}_d)\prod_{k=1}^n\de\tau_k
\end{align*}
where
\begin{align}\label{eq:I_n}
\mathcal{I}_n(\underline{\alpha}_d;{\bf \underline\tau}_d)&:= \prod_{k=1}^{n+1}\left\{\int_{\Lambda}\exp\left\{ic\tau_k<\underline{\alpha}_d,{\bf \underline{V}}_k> \right\}\varphi(\underline{\theta}_{d-1}^k)\prod_{j=1}^{d-2}\de\theta_
j^k\de\phi^k\right\}
\end{align}
where $\Lambda:=[0,\pi]^n\times[0,2\pi]$.

It is well-known that the integral $\mathcal{I}_n(\underline{\alpha}_d;{\bf \underline\tau}_d)$ (see Theorem 2.1 of Orsingher and De Gregorio, 2007, and formula (2.5) of De Gregorio and Orsingher, 2012) is equal to
\begin{equation}\label{intangle}
\mathcal{I}_n(\underline{\alpha}_d;{\bf \underline\tau}_d)=\left\{2^{\frac d2-1}\Gamma\left(\frac d2\right)\right\}^{n+1}\prod_{k=1}^{n+1}\frac{J_{\frac d2-1}(c\tau_k||\underline{\alpha}_d||)}{(c\tau_k||\underline{\alpha}_d||)^{\frac d2-1}}
\end{equation}
and this leads to \eqref{eq:cf}.

Now, by inverting the characteristic function \eqref{eq:cf}, we are able to show that the 
density of the process $\{\underline{\bf X}_d(t),t>0\},$ is given by \eqref{eq:densfree}. Let us denote by $\underline{v}_d$ the vector
$$\underline{v}_d=\left(
 \begin{array}{l} 
   \sin\theta_{1}\sin\theta_{2}\cdot\cdot\cdot\sin\theta_{d-2}\sin\phi \\
          \sin\theta_{1}\sin\theta_{2}\cdot\cdot\cdot\sin\theta_{d-2}\cos\phi \\
      ...\\     
      \sin\theta_{1}\cos\theta_{2}\\
        \cos\theta_{1}

   \end{array}    
    \right).$$
Therefore, by inverting the characteristic function \eqref{eq:cf} and by passing to $d$-dimensional spherical coordinates, we have, for $\underline{{\bf x}}_d\in B_{ct}^d$, that
\begin{align*}
&p_n(\underline{\bf x}_d,t)\\
&=\frac{1}{(2\pi)^d}\int_{\mathbb{R}^d}e^{-i<\underline{\alpha}_d,\underline{\bf x}_d>}\mathcal F_n(\underline{\bf \alpha}_d)\prod_{k=1}^d\de\alpha_k\notag\\
&=\frac{1}{(2\pi)^d}\int_0^\infty \rho^{d-1}\de\rho\int_{\Lambda} e^{-i\rho<\underline{v}_d,\underline{\bf x}_d>}\de\mathbb {S}_1^{d-1}\left\{2^{\frac d2-1}\Gamma\left(\frac d2\right)\right\}^{n+1}\int_{S_n}g({\bf \underline\tau}_n; t)\prod_{k=1}^{n+1}\left\{\frac{J_{\frac d2-1}(c\tau_k\rho)}{(c\tau_k\rho)^{\frac d2-1}} \right\}\prod_{k=1}^{n}\de\tau_k.
\end{align*}
where $\de\mathbb{S}_{1 }^{d-1}:=\prod_{i=1}^{d-2}\{\sin^{d-i-1}\theta_i\de\theta_i\}\de\phi.$ 
By applying formula (2.15) of Orsingher and De Gregorio (2007) 
\begin{equation}\label{eq: (2.12) DGO12}
\int_{\Lambda} e^{-i\rho<\underline{v}_d,\underline{\bf x}_d>}\de\mathbb {S}_1^{d-1}=(2\pi)^{\frac d2}\frac{J_{\frac d2-1}(\rho ||\underline{\bf x}_d||)}{(\rho ||\underline{\bf x}_d||)^{\frac d2-1}},
\end{equation}
we arrive at the claimed result.

\end{proof}

In the analysis of random flights, a central role is played by the probabilistic assumptions on the displacements $c{\bf \underline\tau}_n$. It is useful to assume that  the random vector ${\bf \underline\tau}_n$ has the following Dirichlet distribution    
\begin{equation}\label{eq:dirichlet}
g({\bf \underline\tau}_n;\underline{\bf q}_{n+1},t)=\frac{\Gamma\left((n+1)\left(\frac dh-1\right)\right)}{(\Gamma\left(\frac dh-1\right))^{n+1}}\frac{1}{t^{(n+1)\left(\frac dh-1\right)}}\prod_{k=1}^{n+1}\tau_k^{\frac dh-2},
\end{equation}
where ${\bf \underline\tau}_n\in S_n$, with parameters $\underline{\bf q}_{n+1}:=(\frac dh-1,...,\frac dh-1),h:=1,2,$ and the conditions $d\geq 2$ if $h=1$ and $d\geq 3$ if $h=2$ hold. Under the assumptions \eqref{eq:dirichlet}, the density function \eqref{eq:densfree} can be evaluated explicitly   
\begin{equation}\label{eq:conddensfree}
p_{n,h}(\underline{{\bf x}}_d,t)=\frac{\Gamma(\frac{n+1}{2}(d-h)+\frac h2)}{\Gamma(\frac n2(d-h))}\frac{(c^2t^2-||\underline{{\bf x}}_d||)^{\frac n2(d-h)-1}}{\pi^{\frac d2}(ct)^{(n+1)(d-h)+h-2}}{\bf 1}_{B_R^d}(\underline{{\bf x}}_d),
\end{equation}
while for the process $\{D_d(t)=||\underline{\bf X}_d(t)||,t>0\},$ we are able to obtain the following conditional density
\begin{equation}\label{eq:densradialfree}
q_{n,d,h}(r,t)=\frac{2\Gamma(\frac{n+1}{2}(d-h)+\frac h2)}{\Gamma(\frac d2)\Gamma(\frac n2(d-h))}\frac{r^{d-1}(c^2t^2-r^2)^{\frac n2(d-h)-1}}{(ct)^{(n+1)(d-h)+h-2}}{\bf 1}_{(0,R)}(r),
\end{equation}
(see De Gregorio and Orsingher (2012) and Le Ca\"er (2010)). We observe that $\{D_d(t),t>0\},$ is related to the Beta distribution. Indeed, $D_d(t)/c^2t^2$ is distributed as a Beta random variable with parameters $\frac d2$ and $\frac n2(d-h)$. 

\section{Reflecting random flights in spheres}
\subsection{Definition and probability distributions}

Let us consider a random flight  $\{\underline{{\bf X}}_d(t),t>0\},$ defined as in the Introduction. When a sample path of $\underline{{\bf X}}_d(t)$ strikes the sphere $\se^{d-1}_R$, the trajectory of the process is reflected inside $\se^{d-1}_R$. The reflection of the random flight on the boundary $\se_R^{d-1}$ can be envisaged in different ways. The specular reflection (the incoming sample path forms the same angle as the reflected trajectory with respect to the normal vector to $\se_R^{d-1}$) seems the most natural one. Nevertheless, from the mathematical point of view, the reflecting surface is closed and this implies that the probability distribution of the reflected process takes a cumbersome form for sufficiently large values of $t$.  For this reason, we assume that the reflection is based on the principle of circular inversion in spheres defined in Appendix \ref{refspheres}. The most important effect of this procedure is that the sample paths obtained by reflection are deformed and take the structure of circumference arcs. This leads to a new process, namely, the reflecting random flight moving in $B^{d}_R\cup \se^{d-1}_R$.

\begin{definition}\label{def:refsp}
The reflecting random flight $\{\underline{{\bf X}}_d^*(t),t>0\},$ reflected on the sphere $\se_R^{d-1}$ is constructed by means of the free process $\{\underline{{\bf X}}_d(t),t>0\}$ as follows: 1) if $t\leq \frac Rc,$ then $\underline{{\bf X}}_d^*(t)
:=\underline{{\bf X}}_d(t)$; 2) if $t>\frac Rc,$ then if at least one change of direction happens during the time interval $[0,t]$, we have that
\begin{align}\label{eq:def}
\underline{{\bf X}}_d^*(t)
:=\underline{{\bf X}}_d(t){\bf 1}_{B_R^d}(\underline{{\bf X}}_d(t))+\mu_R(\underline{{\bf X}}_d(t))
{\bf 1}_{C^d_{R,ct}}(\underline{{\bf X}}_d(t)),
\end{align}
where $\mu_R(\underline{{\bf x}}_d)=R^2\frac{\underline{{\bf x}}_d}{||\underline{{\bf x}}_d||^2}$ is the inversion map defined by \eqref{eq:appendixmap1}; while if there are no deviations
\begin{align}\label{eq:def2}
\underline{{\bf X}}_d^*(t)
:=\mu_R(\underline{{\bf X}}_d(t))
{\bf 1}_{\se^{d-1}_{ct}}(\underline{{\bf X}}_d(t)).
\end{align}
\end{definition}

Definition \ref{def:refsp} leads to the following considerations on the sample paths of $\{\underline{{\bf X}}_d^*(t),t>0\}$:

\begin{itemize}
\item the reflecting random flight has two components: the first one is given by the free process $\{\underline{{\bf X}}_d(t),t>0\}$; the second component $\mu_R(\underline{{\bf X}}_d(t))$ is due to the reflection in $\se_R^{d-1}$ of the sample paths of $\{\underline{{\bf X}}_d(t),t>0\},$ wandering outside the sphere $\se_R^{d-1}$. It is worth mentioning that the property  P4) implies that the reflected paths have opposite orientation w.r.t. the trajectories of $\{\underline{{\bf X}}_d(t),t>0\}$ moving outside $\se_R^{d-1}$;
\item the reflecting component of $\{\underline{{\bf X}}_d^*(t),t>0\}$ is given by
\begin{align}\label{eq:refcomp}
\mu_R(\underline{{\bf X}}_d(t))=\frac{R^2}{||\underline{{\bf X}}_d(t)||^2}\underline{{\bf X}}_d(t)=\frac{cR^2}{||\underline{{\bf X}}_d(t)||^2}\sum_{k=1}^{n+1}{\bf \underline{V}}_k\tau_k,
\end{align}
for $\underline{{\bf X}}_d(t)\in C^d_{R,ct}$. Therefore, $\frac{cR^2}{||\underline{{\bf X}}_d(t)||^2}$ can be thought of as the random velocity of \eqref{eq:refcomp}. Since $\frac{cR^2}{||\underline{{\bf X}}_d(t)||^2}\leq c$, the reflected motion travels more slowly than the one related to the process $\{\underline{{\bf X}}_d(t),t>0\}$. For instance, see sample path ${\bf b}$ in Figure \ref{fig1};
\item the property P5) of the map $\mu_R$ implies that the sample paths of $\{\underline{{\bf X}}_d^*(t),t>0\}$ can be represented by broken lines, when no reflection has occurred, and by the composition of straight lines and circumference arcs, when at least one reflection has taken place (see sample path a in Figure \ref{fig1}). 
\end{itemize}

\begin{figure}[t]
\begin{center}
\includegraphics[angle=0,width=0.7\textwidth]{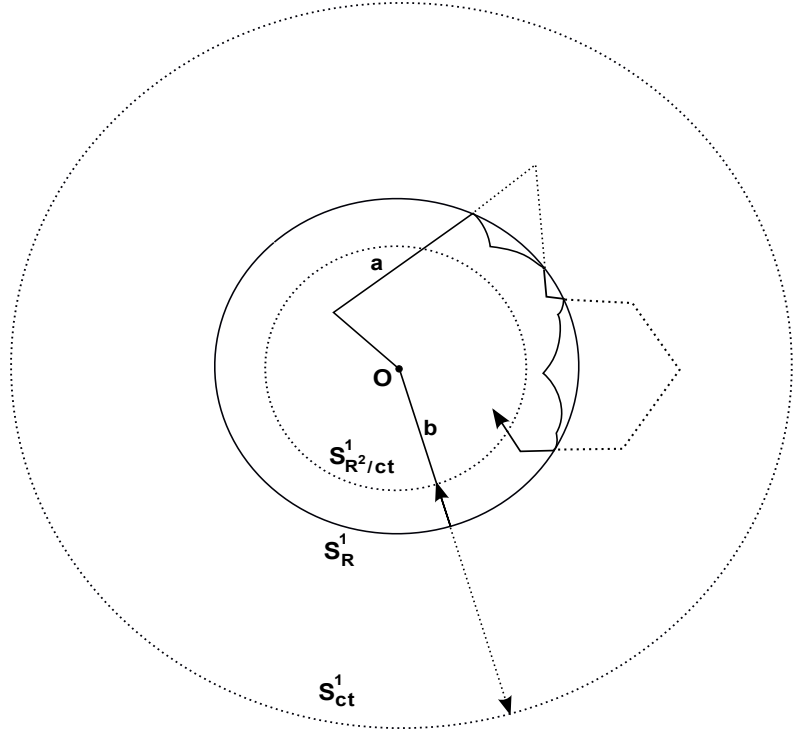}
\caption{Two typical sample paths are depicted. The trajectory ${\bf a}$ is composed of straight lines and arc of spheres, while the sample path ${\bf b}$ is obtained if no changes of direction happen in the interval $[0,t]$.}\label{fig1}
\end{center}
\end{figure}

 We are able to provide the conditional probability distributions $\{\underline{{\bf X}}_d^*(t),t>0\},$ by means of those of  $\{\underline{{\bf X}}_d(t),t>0\}$. Therefore, by means of \eqref{eq:densfree}, we are able to obtain the next result.
\begin{theorem}
If  $\mathcal N(t)=n$, with $n\geq 1,$ the process $\{\underline{{\bf X}}_d^*(t),t>0\},$ has the following conditional density function
\begin{align}\label{eq:condens}
p_{n}^*(\underline{{\bf x}}_d,t)=\begin{cases}
p_n(\underline{{\bf x}}_d,t){\bf 1}_{B_R^d}(\underline{{\bf x}}_d),& t\leq\frac{R}{c}\\
p_n(\underline{{\bf x}}_d,t){\bf 1}_{B_R^d}(\underline{{\bf x}}_d)+\frac{R^{2d}}{||\underline{{\bf x}}_d||^{2d}}p_n\left(R^2\frac{\underline{{\bf x}}_d}{||\underline{{\bf x}}_d||^2},t\right){\bf 1}_{C^d_{\frac{R^2}{ct},R}}(\underline{{\bf x}}_d),& t>\frac{R}{c}.
\end{cases}
\end{align}
where $p_n(\underline{{\bf x}}_d,t)$ is equal to \eqref{eq:densfree}.
\end{theorem}
\begin{proof}
The case $ t\leq\frac{R}{c}$ is obvious. Now we assume that $t>\frac Rc$. Let $A$ be a Borel set such that $A\cap \overline{B_{R}^d}\neq \varnothing$. Let $\underline{{\bf Y}}_d(t):=\mu_R(\underline{{\bf X}}_d(t))=R^2\frac{\underline{{\bf X}}_d(t)}{||\underline{{\bf X}}_d(t)||^2},$ we get 
\begin{align}\label{eq:step1dens}
P_n\{\underline{{\bf X}}_d^*(t)\in A\}&=P_n\{\underline{{\bf X}}_d(t)\in A\cap B_{R}^d\}+P_n\left\{\underline{{\bf Y}}_d(t)\in A\cap  C_{\frac{R^2}{ct},R}^d\right\}\notag\\
&=\int_{ A\cap B_{R}^d}p_n(\underline{{\bf x}}_d,t)\prod_{k=1}^d\de x_k+P_n\left\{\underline{{\bf Y}}_d(t)\in A\cap C^d_{\frac{R^2}{ct},R}\right\}.
\end{align} 
The first term in the previous expression refers to the sample paths which do not attain the boundary $\se_R^{d-1}$ up to time $t,$ while the second probability concerns the trajectories reflecting on $\se_R^{d-1}$.
By setting $$D_A:=\mu_R^{-1}\left(A\cap  C_{\frac{R^2}{ct},R}^d\right)=\left\{\underline{{\bf y}}_d\in \re^d: \underline{{\bf x}}_d=\mu_R(\underline{{\bf y}}_d)\in A\cap C^d_{\frac{R^2}{ct},R}\right\}$$
then, by means of Jacobi transformation formula, we have that
\begin{align}\label{eq:step2dens}
P_n\left\{\underline{{\bf Y}}_d(t)\in A\cap C^d_{\frac{R^2}{ct},R}\right\}&=P_n\{\underline{{\bf X}}_d(t)\in D_A\}\notag\\
&=\int_{D_A}p_n(\underline{{\bf y}}_d,t)\prod_{k=1}^d\de y_k\notag\\
&=\int_{A\cap C^d_{\frac{R^2}{ct},R}} p_n\left(\mu_R^{-1}(\underline{{\bf x}}_d),t\right)|\text{det}(J_{\mu_R^{-1}}(\underline{{\bf x}}_d))|\prod_{k=1}^d \de x_k\notag\\
&=\int_{A\cap C^d_{\frac{R^2}{ct},R}} p_n\left(\mu_R(\underline{{\bf x}}_d),t\right)|\text{det}(J_{\mu_R}(\underline{{\bf x}}_d))|\prod_{k=1}^d \de x_k
\end{align}
where in the last step we have used the following facts: $\mu_R=\mu_R^{-1}$ (which follows from the property P3)) and $|\text{det}(J_{\mu_R}(\underline{{\bf x}}_d))|=\frac{R^{2d}}{||\underline{{\bf x}}_d||^{2d}}$. From \eqref{eq:step1dens} and \eqref{eq:step2dens}, the result \eqref{eq:condens} immediately follows.

\end{proof}

From Lemma \ref{lemma1} emerges that the random process $\{\underline{\bf X}_d(t),t>0\}$ is isotropic, namely $Q(\underline{\bf X}_d(t))\sim \underline{\bf X}_d(t)$ for all $Q\in O(d)$ and $p_n(\underline{{\bf x}}_d,t)=p_n(||\underline{{\bf x}}_d||,t)$. Furthermore, from \eqref{eq:densfree}, we have that $\{\underline{\bf X}_d^*(t),t>0\}$  is invariant by rotation as well. Therefore, $p_n^*(\underline{{\bf x}}_d,t)$ depends on the distance $||\underline{{\bf x}}_d||$ and thus we can write 
\begin{align}\label{eq:rotinv}
p_n^*(\underline{{\bf x}}_d,t)&=p_n^*(||\underline{{\bf x}}_d||,t)\notag\\
&=p_n(||\underline{{\bf x}}_d||,t){\bf 1}_{(0,R)}(||\underline{{\bf x}}_d||)+\frac{R^{2d}}{||\underline{{\bf x}}_d||^{2d}}p_n\left(\frac{R^2}{||\underline{{\bf x}}_d||},t\right){\bf 1}_{(R^2/ct,R]}(||\underline{{\bf x}}_d||)
\end{align}
for $t>\frac Rc.$
\begin{definition}
The reflecting radial process $\{D_d^*(t),t>0\},$ represents the Euclidean distance from $\underline{{\bf 0}}_d$ of the position $\underline{{\bf X}}_d^*(t)$, namely $D_d^*(t)=||\underline{{\bf X}}_d^*(t)||$. It can be defined by
\begin{align}
D_d^*(t)=D_d(t){\bf 1}_{(0,R)}(D_d(t))+\frac{R^2}{D_d(t)}{\bf 1}_{[R,ct]}(D_d(t)),
\end{align}
where $D_d(t) =||\underline{{\bf X}}_d(t)||$. 
\end{definition}

 As a consequence of the isotropic structure of $\{\underline{\bf X}_d^*(t),t>0\}$, the conditional density function of $\{D_d^*(t),t>0\},$ becomes
\begin{equation}\label{eq:distancecircinv}
q_{n,d}^*(r,t)=r^{d-1}\text{area}(\mathbb{S}_{1}^{d-1})p_n^*(r,t){\bf 1}_{(0,R]}(r).
\end{equation}

\subsection{Reflecting Dirichlet random flights}

A suitable choice of the distribution $g({\bf \underline\tau}_n; t)$ leads to explicit expressions for the conditional density functions \eqref{eq:condens}. As we have seen in Section \ref{sec:prelim} the Dirichlet distributions play a special role in the study of random flights. The assumption \eqref{eq:dirichlet} and the results \eqref{eq:conddensfree} and \eqref{eq:densradialfree} imply that the reflecting process  $\{\underline{{\bf X}}_d^*(t),t>0\}$ has probability law \eqref{eq:condens} given by 
\begin{align}\label{eq:densrefl}
p_{n,h}^*(\underline{{\bf x}}_d,t)&=\frac{\Gamma(\frac{n+1}{2}(d-h)+\frac h2)}{\Gamma(\frac n2(d-h))}\frac{1}{\pi^{\frac d2}(ct)^{(n+1)(d-h)+h-2}}\\
&\quad\times\left[(c^2t^2-||\underline{{\bf x}}_d||^2)^{\frac n2(d-h)-1}{\bf 1}_{B_R^d}(\underline{{\bf x}}_d)+\frac{R^{2d}}{||\underline{{\bf x}}_d||^{2d}}\left(c^2t^2-\frac{R^4}{||\underline{{\bf x}}_d||^2}\right)^{\frac n2(d-h)-1}{\bf 1}_{C^d_{\frac{R^2}{ct},R}}(\underline{{\bf x}}_d)\right],\notag
\end{align}
with $t>\frac Rc$, while $\{D_d^*(t),t>0\}$ has the following conditional distribution
\begin{align}\label{eq:radialcondensdir}
q_{n,d,h}^*(r,t)&=\frac{2\Gamma(\frac{n+1}{2}(d-h)+\frac h2)}{\Gamma(\frac d2)\Gamma(\frac n2(d-h))(ct)^{(n+1)(d-h)+h-2}}\\
&\quad\times\left[r^{d-1}(c^2t^2-r^2)^{\frac n2(d-h)-1}{\bf 1}_{(0,R)}(r)+\frac{R^{2d}}{r^{d+1}}\left(c^2t^2-\frac{R^4}{r^2}\right)^{\frac n2(d-h)-1}{\bf 1}_{(R^2/ct,R]}(r)\right],\notag
\end{align}
with $t>\frac Rc$. In this case we call these random walks ``reflecting Dirichlet random flights''.

\begin{remark}\label{remarkun}
The probability $P_n\{\underline{{\bf X}}_d(t)\in \de \underline{{\bf x}}_d\}$ is uniform inside the ball $B_{ct}^d$ in the following cases: (i) $h=1, d=2, n=2$; (ii) $h=1, d=3, n=1$;  (iii) $h=2, d=3, n=2$; (iv) $h=2, d=4, n=1$.

In the cases (i)-(iv), we have that the function $p_{n,h}^*(\underline{{\bf x}}_d,t)$ becomes
\begin{align*}
p_{2,1}^*(\underline{{\bf x}}_2,t)&=\frac{1}{\pi(ct)^2}\left[{\bf 1}_{B_R^2}(\underline{{\bf x}}_2)+\frac{R^{4}}{||\underline{{\bf x}}_2||^{4}}{\bf 1}_{C^2_{\frac{R^2}{ct},R}}(\underline{{\bf x}}_2)\right],
\\
p_{1,1}^*(\underline{{\bf x}}_3,t)&=p_{2,2}^*(\underline{{\bf x}}_3,t)=\frac{\Gamma(\frac 52)}{\pi^{\frac 32}(ct)^3}\left[{\bf 1}_{B_R^3}(\underline{{\bf x}}_3)+\frac{R^{6}}{||\underline{{\bf x}}_3||^{6}}{\bf 1}_{C^3_{\frac{R^2}{ct},R}}(\underline{{\bf x}}_3)\right],
\\
p_{1,2}^*(\underline{{\bf x}}_4,t)&=\frac{\Gamma(3)}{\pi^2(ct)^4}\left[{\bf 1}_{B_R^4}(\underline{{\bf x}}_4)+\frac{R^{8}}{||\underline{{\bf x}}_4||^{8}}{\bf 1}_{C^4_{\frac{R^2}{ct},R}}(\underline{{\bf x}}_4)\right].
\end{align*}
This implies that the reflecting random flights are never uniformly distributed inside the ball $B_{R}^d$.
\end{remark}

Now, we focus our attention on the distribution function of $\{D_d^*(t),t>0\}$ which also provides information on the probability of the position of the reflecting random flight at time $t>0$, that is
$$P_n\{D_d^*(t)<r\}=P_n\{\underline{\bf X}_d^*(t)\in B_r^d\}, \quad 0<r\leq R.$$ 
The probability distribution of the distance process $\{D_d^*(t),t>0\}$ is related to the Beta distribution as shown in the next result. Let $r\in(0,R]$ and $t>\frac Rc$,
by using \eqref{eq:densradialfree}, we immediately obtain that
\begin{align}\label{eq:probball2}
P_n\{ D_d^*(t)<r\}&=\int_0^{\frac{r^2}{(ct)^2}}\frac{x^{\frac d2-1}(1-x)^{\frac n2(d-h)-1}}{B\left(\frac d2,\frac n2(d-h)\right)}\de x+\left[\int_{\frac{R^4}{(ctr)^2}}^1\frac{x^{\frac d2-1}(1-x)^{\frac n2(d-h)-1}}{B\left(\frac d2,\frac n2(d-h)\right)}\de x\right]{\bf 1}_{(R^2/ct,R]}(r)\notag\\
&=\frac{B\left(\frac{r^2}{(ct)^2};\frac d2,\frac n2(d-h)\right)}{B\left(\frac d2,\frac n2(d-h)\right)}+\left[1-\frac{B\left(\frac{R^4}{(ctr)^2};\frac d2,\frac n2(d-h)\right)}{B\left(\frac d2,\frac n2(d-h)\right)}\right]{\bf 1}_{(R^2/ct,R]}(r)
\end{align}
where $B(a,b):=\Gamma(a)\Gamma(b)/\Gamma(a+b)$
and $B(z;a,b):=\int_0^zu^{a-1}(1-u)^{b-1}du$ is the incomplete gamma function. From \eqref{eq:probball2} we obtain that
\begin{align*}
P_n\{\underline{\bf X}_d^*(t)\in C^d_{\frac{R^2}{ct},R}\}&=P_n\left\{\frac{R^2}{ct}<D_d^*(t)\leq R\right\}\\
&=1-B\left(R^4;\frac d2,\frac n2(d-h)\right)
\end{align*}
The  probability \eqref{eq:probball2} becomes particularly simple for $d=2$ and $h=1$. Indeed, we have that
 
\begin{align}
P_n\{ D_2^*(t)<r\}&=1-\left(1-\frac{r^2}{(ct)^2}\right)^{\frac n2}+\left(1-\frac{R^4}{(ctr)^2}\right)^{\frac n2}{\bf 1}_{(R^2/ct,R]}(r),\quad r\in(0,R].
\end{align}

 Furthermore, if we assume that $h=2, d=2d',d'\geq 2,$ it is possible to write down $P_n\{ D_d^*(t)<r\}$ by means of the probability distribution of binomial r.v.'s. By exploiting the following well-known result
 \begin{equation}\label{eq:relincombeta}
\frac{B(x;a,b)}{B(a,b)}=\sum_{k=a}^{a+b-1}\binom{a+b-1}{k}x^k(1-x)^{a+b-1-k},\quad a,b\in\mathbb N,
\end{equation}
it is not hard to prove that 
\begin{align}
P_n\{D_d^*(t)<r\}&=P\left\{\frac d2\leq Y_{n,d}\leq\frac{n+1}{2}(d-2)\right\}+P\left\{0\leq \hat Y_{n,d}<\frac{d}{2}\right\}{\bf 1}_{(R^2/ct,R]}(r)
\end{align}
where $Y_{n,d}\sim$Bin$\left(\frac{n+1}{2}(d-2),\frac{r^2}{(ct)^2}\right)$ and $\hat Y_{n,d}\sim$Bin$\left(\frac{n+1}{2}(d-2),\frac{R^4}{(ctr)^2}\right)$.

For  $t>\frac Rc$, we are also able to derive the $m$-th moment, with $m\geq 1,$ of $\{D_d^*(t),t>0\}$. We have that
\begin{align}\label{eq:moments}
E_n\{D_d^*(t)\}^m&=\frac{1}{B\left(\frac d2,\frac n2(d-h)\right)}\left[(ct)^m\int_0^{\frac{R^2}{(ct)^2}}x^{\frac {d+m}{2}-1}(1-x)^{\frac{n}{2}(d-h)-1}\de x\right.\notag\\
&\quad\left.+\left(\frac{R^2}{ct}\right)^{m}\int_{\frac{R^2}{(ct)^2}}^1x^{\frac {d-m}{2}-1}(1-x)^{\frac{n}{2}(d-h)-1}\de x\right].
\end{align}

Moreover, if $d>m,$ we can write \eqref{eq:moments} in terms of beta and incomplete beta functions as follows
\begin{align}\label{eq:moments2}
E_n\{D_d^*(t)\}^m &=(ct)^m\frac{B\left(\frac{R^2}{(ct)^2};\frac {d+m}{2},\frac n2(d-h)\right)}{B\left(\frac {d}{2},\frac n2(d-h)\right)}+\left(\frac{R^2}{ct}\right)^{m}\left[1-\frac{B\left(\frac{R^2}{(ct)^2};\frac{ d-m}{2},\frac n2(d-h)\right)}{B\left(\frac d2,\frac n2(d-h)\right)}\right].\end{align}

\subsection{Random flights and the Euler-Poisson-Darboux equation}

In De Gregorio and Orsingher (2012) (Remark 2.7) is observed that the functions
\begin{equation}\label{eq:epd}
f_\beta(\underline{{\bf x}}_d,t):=(c^2t^2-||\underline{{\bf x}}_d||^2)^\beta,\quad ||\underline{{\bf x}}_d||<ct, \beta\in\re,
\end{equation}
satisfy the following Euler-Poisson-Darboux (EPD) partial differential equation
\begin{equation}\label{eq:epdcon}
\frac{\partial ^2u}{\partial t^2}-\frac{2\beta-1+d}{t}\frac{\partial u}{\partial t}=c^2\Delta u
\end{equation}
where $\Delta:=\sum_{k=1}^d\frac{\partial^2}{\partial x_k^2}$. If $\beta=\frac{n}{2}(d-h)-1,$ we have that 
$$f_{\frac{n}{2}(d-h)-1}(\underline{{\bf x}}_d,t)=\frac{\pi^{\frac d2}(ct)^{(n+1)(d-h)+h-2}\Gamma(\frac n2(d-h))}{\Gamma(\frac{n+1}{2}(d-h)+\frac h2)}p_{n,h}(\underline{{\bf x}}_d,t).$$
Therefore, by exploiting the equation \eqref{eq:epdcon}, it is not hard to show that $p_{n,h}(\underline{{\bf x}}_d,t)$ is a solution of the following EPD equation
\begin{equation}\label{eq:epdcon2}
\frac{\partial ^2u}{\partial t^2}+\frac{(n+1)(d-h)+h-1}{t}\frac{\partial u}{\partial t}=c^2\Delta u.
\end{equation}
The projection of the random process $\{\underline{{\bf X}}_d(t),t>0\}$ onto a lower space of dimension $m,$ implies that the conditional marginal distributions become
\begin{equation*}
p_{n,h}^d(\underline{{\bf x}}_m,t)=\frac{\Gamma(\frac{n+1}{2}(d-h)+\frac h2)}{\Gamma(\frac{ n+1}{2}(d-h)+\frac{h-m}{2})}\frac{(c^2t^2-||\underline{{\bf x}}_m||)^{\frac{n+1}{2}(d-h)-\frac{m-h}{2}-1}}{\pi^{\frac m2}(ct)^{(n+1)(d-h)+h-2}},
\end{equation*}
(see formulas (2.26) and (2.27) of De Gregorio and Orsingher, 2012). By means of the same considerations used for the density function $p_{n,h}(\underline{{\bf x}}_d,t)$ we obtain that $p_{n,h}^d(\underline{{\bf x}}_m,t)$ is still solution of the EPD equation \eqref{eq:epdcon}.

In the same spirit of the previous considerations, the function
$$\bar f_\beta(\underline{{\bf x}}_d,t):=\left(c^2t^2-	\frac{R^4}{||\underline{{\bf x}}_d||^2}\right)^\beta,\quad\frac{R^2}{ct}< ||\underline{{\bf x}}_d||\leq R, \beta\in\re,
$$
solves the following EPD partial differential equation with time and space varying coefficients
\begin{equation}\label{eq:epd2}
\frac{\partial ^2u}{\partial t^2}-\frac{a_\beta(\underline{{\bf x}}_d)}{t}\frac{\partial u}{\partial t}=c^2\Delta u,\end{equation}
where $$a_\beta(\underline{{\bf x}}_d):=2\beta-1+\frac{R^4}{||\underline{{\bf x}}_d||^2}\left[\frac{4-d}{||\underline{{\bf x}}_d||^2}+2(\beta-1)\left(\frac{1-R^4/||\underline{{\bf x}}_d||^4}{c^2t^2-R^4/||\underline{{\bf x}}_d||^2}\right)\right].$$ For $d=4$ and $\beta=1$, we have that the function $\bar f_1(\underline{{\bf x}}_4,t)$
satisfies the equation
$$\frac{\partial ^2u}{\partial t^2}-\frac{1}{t}\frac{\partial u}{\partial t}=c^2\Delta u.$$ 

 By setting 
 $$\bar p_{n,h}(\underline{{\bf x}}_d,t):=\frac{\Gamma(\frac{n+1}{2}(d-h)+\frac h2)}{\Gamma(\frac n2(d-h))}\frac{\frac{R^{2d}}{||\underline{{\bf x}}_d||^{2d}}}{\pi^{\frac d2}(ct)^{(n+1)(d-h)+h-2}}\left(c^2t^2-\frac{R^4}{||\underline{{\bf x}}_d||^2}\right)^{\frac n2(d-h)-1},$$
 with $\frac{R^2}{ct}< ||\underline{{\bf x}}_d||\leq R,$
 we have that 
$$\bar f_{\frac{n}{2}(d-h)-1}(\underline{{\bf x}}_d,t)=\frac{\pi^{\frac d2}(ct)^{(n+1)(d-h)+h-2}\Gamma(\frac n2(d-h))}{\Gamma(\frac{n+1}{2}(d-h)+\frac h2)}\frac{||\underline{{\bf x}}_d||^{2d}}{R^{2d}}\bar p_{n,h}(\underline{{\bf x}}_d,t).$$
Therefore, by exploiting the equation \eqref{eq:epd2}, we conclude that $\bar p_{n,h}(\underline{{\bf x}}_d,t)$ is a solution of the following partial differential equation

\begin{align}\label{eq:pderef}
&\frac{\partial ^2u}{\partial t^2}+\frac{2(2\beta+d)-a_\beta(\underline{{\bf x}}_d)}{t}\frac{\partial u}{\partial t}+\left[\frac{(2\beta+d)((2\beta+d-1)-a_\beta(\underline{{\bf x}}_d))}{t^2}\right]u\notag\\
&=c^2\left[\Delta +\frac{2d(3d-2)}{||\underline{{\bf x}}_d||^2}+\frac{4d\langle \underline{{\bf x}}_d,\nabla \rangle}{||\underline{{\bf x}}_d||^2}\right]u,
\end{align}
where $\beta=\frac{n}{2}(d-h)-1$ and $\nabla u:=$grad$u$. The equation \eqref{eq:pderef} no longer has the structure of the EPD equation.

\subsection{On the unconditional probability distributions}

In order to obtain unconditional densities for $\{\underline{\bf X}_d^*(t),t>0\},$ we should specify the probability distribution of $\mathcal N(t)$. Different choices of the above probability law lead to different unconditional densities of the reflecting random flight. We assume that the number of deviations $\mathcal{N}(t)=\mathcal{N}_d(t), d\geq 2$, at time $t>0$, possesses the distribution of a weighted Poisson random variable. Let $\{N(t),t>0\}$ be a homogeneous Poisson process with rate $\lambda>0,$ a random variable $\mathcal{N}_d(t)$ has weighted Poisson probability distribution if
\begin{equation}
P\{\mathcal{N}_d(t)=n\}=\frac{w_n P\{N(t)=n\}}{\sum_{k=0}^\infty w_k P\{N(t)=k\}}, \quad n\in\mathbb N_0,
\end{equation}
where $w_ks$ are non-negative weight functions with $0<\sum_{k=0}^\infty  w_k P\{N(t)=k\}<\infty$ (for more details see Balakrishnan and Kozubowski, 2008). In our context, we choose 
the weights as follows
$$w_k=\frac{k!}{\Gamma((\frac{d-h}{2})k+\frac d2)}$$
and then for this choice we obtain that
\begin{equation}\label{lawgenpoi}
P\left\{\mathcal{N}_d(t)=n\right\}=\frac{1}{E_{\frac{d-h}{2},\frac d2}(\lambda t)}\frac{(\lambda t)^n}{\Gamma((\frac{d-h}{2})n+\frac d2)},\quad n\in\mathbb{N}_0,
\end{equation}
where $E_{\alpha,\beta}(x)=\sum_{k=0}^\infty \frac{x^k}{\Gamma(\alpha k+\beta)},\, x\in\mathbb{R},\alpha,\beta>0,$ is the two-parameter Mittag-Leffler function. The random variable $\mathcal{N}_d(t)$ with probability distribution \eqref{lawgenpoi} coincides with (4.1) in Beghin and Orsingher (2009).
\begin{theorem}
For $t>\frac Rc$ and by assuming \eqref{lawgenpoi},  we have that the unconditional probability distribution of $\{\underline{\bf X}_{d}^*(t),t>0\}$ becomes
\begin{equation}\label{eq:comprdis}
P\{\underline{\bf X}_{d}^*(t)\in \de\underline{\bf x}_{d}\}=p_h^*(\underline{\bf x}_d,t)\prod_{k=1}^d\de x_k+\frac{1}{E_{\frac{d-h}{2},\frac d2}(\lambda t)\Gamma(\frac d2)}\mu^*(\de\underline{\bf x}_d )
\end{equation}
where
\begin{align}\label{eq:unconddensrefl0}
p_h^*(\underline{\bf x}_d,t)&=\frac{1}{(ct)^d\pi^{\frac d2}}\frac{1}{E_{\frac{d-h}{2},\frac d2}(\lambda t)}\left\{\left[\gamma_h(||\underline{\bf x}_d||,t)\right]^{1-\frac{2}{d-h}}E_{\frac {d-h}{2},\frac {d-h}{2}}\left(\gamma_h(||\underline{\bf x}_d||,t)\right){\bf 1}_{B_R^d}(\underline{{\bf x}}_d)\right.\notag\\
&\quad+\left.\frac{R^{2d}}{||\underline{{\bf x}}_d||^{2d}}\left[\gamma_h\left(\frac{R^2}{||\underline{\bf x}_d||},t\right)\right]^{1-\frac{2}{d-h}}E_{\frac {d-h}{2},\frac {d-h}{2}}\left(\gamma_h\left(\frac{R^2}{||\underline{\bf x}_d||},t\right)\right){\bf 1}_{C^d_{\frac{R^2}{ct},R}}(\underline{{\bf x}}_d)\right\},
\end{align}
with 
$$\gamma_h(||\underline{\bf x}_d||,t):=\lambda t\left(1-\frac{||\underline{\bf x}_{d}||^2}{c^2t^2}\right)^{\frac {d-h}{2}},$$
while $\mu^*$ is the uniform distribution on $\se^{d-1}_{\frac{R^2}{ct}}.$

\end{theorem}

\begin{proof}
The second term in \eqref{eq:comprdis} arises from the following considerations. The probability distribution of the random flight $\{\underline{\bf X}_d^*(t),t>0\}$ admits a singular component emerging in the case $\mathcal{N}_d(t)=0,$ that is if the initial direction of the motion does not change up to time $t$. For $t>\frac Rc,$ the circular inversion of $\mathbb S_{ct}^{d-1}$ with respect to $\mathbb S_R^{d-1}$ leads to $\mathbb S_{\frac{R^2}{ct}}^{d-1}$.  Therefore, if there are no changes of direction in $(0,t]$, the reflected path reaches $\mathbb S_{\frac{R^2}{ct}}^{d-1}$ and
$$P\left\{\underline{\bf X}_d^*(t)\in\mathbb S_{\frac{R^2}{ct}}^{d-1}\right\}=P\{\mathcal{N}_d(t)=0\}=\frac{1}{E_{\frac{d-h}{2},\frac d2}(\lambda t)\Gamma(\frac d2)}.$$

For the remaining part of the distribution $P\{\underline{\bf X}_{d}^*(t)\in \de\underline{\bf x}_{d}\}$ we can observe that
\begin{align*}
p_h^*(\underline{\bf x}_d,t)&=\sum_{n=1}^\infty p_{n,h}^*(\underline{\bf x}_d,t)P\{\mathcal N_d(t)=n\}
\end{align*}
where $p_{n,h}(\underline{\bf x}_d,t)$ is defined by \eqref{eq:densrefl}. The same steps performed in the proof of Theorem 5 in De Gregorio and Orsingher (2012), lead to the result \eqref{eq:unconddensrefl0}.

\end{proof}

\begin{remark}
For $\frac{d-h}{2}=1,$ that is $d=3$ for $h=1$ and $d=4$ for $h=2$, the function \eqref{eq:unconddensrefl} reduces to
\begin{align*}
p_h^*(\underline{\bf x}_d,t)&=\frac{\lambda }{c^3t^2\pi^{\frac d2}}\frac{1}{E_{1,\frac d2}(\lambda t)}\left\{ \exp\left\{\lambda t\left(1-\frac{||\underline{\bf x}_{d}||^2}{c^2t^2}\right)\right\}{\bf 1}_{B_R^d}(\underline{{\bf x}}_d)\right.\\
&\quad\left.+\frac{R^{2d}}{||\underline{{\bf x}}_d||^{2d}}\exp\left\{\lambda t\left(1-\frac{R^4}{c^2t^2||\underline{{\bf x}}_d||^2}\right)\right\}{\bf 1}_{C^d_{\frac{R^2}{ct},R}}(\underline{{\bf x}}_d)\right\},
\end{align*}
since $E_{1,1}(x)=e^x.$

\end{remark}

We observe that the Dirichlet distribution \eqref{eq:dirichlet} with $h=1$ and $d=2$ or $h=2$ and $d=4$, reduces to the uniform law in $S_n$. Therefore, alternatively to \eqref{lawgenpoi}, we  can assume that 
$$g(\underline{\tau}_n;t)=\frac{n!}{t^n}{\bf1}_{S_n}(\underline{\tau}_n).$$
In other words,  instead of $\mathcal{N}_d(t)$, we can suppose that the changes of direction are governed by a homogeneous Poisson process  $\{N(t),t>0\}$ with rate $\lambda>0$. Under these assumptions, the unconditional distributions of random flights $\{\underline{\bf X}_d(t),t>0\},d=2,4,$ are given by 
\begin{align}\label{eq:dgo1}
   p_h(\underline{\bf x}_d,t)&=\sum_{n=1}^\infty p_{n,h}(\underline{\bf x}_d,t)P\{N(t)=n\}\notag\\
&=
\begin{cases}
\frac{\lambda e^{-\lambda t}}{2\pi c}\frac{e^{\frac{\lambda}{c}\sqrt{c^2t^2-||\underline{\bf x}_2||^2}}}{\sqrt{c^2t^2-||\underline{\bf x}_2||^2}}{\bf 1}_{B_{ct}^2}(\underline{{\bf x}}_2),& d=2,h=1,\\
\frac{\lambda}{c^4t^3\pi^2}e^{-\frac{\lambda}{c^2t}||\underline{\bf x}_4||^2}\left\{2+\frac{\lambda}{c^2t}(c^2t^2-||\underline{\bf x}_4||^2)\right\}{\bf 1}_{B_{ct}^4}(\underline{{\bf x}}_4),& d=4,h=2.
\end{cases}
\end{align}
(for the case $d=2, h=1,$ see (1.2) of Stadje, 1987, (18) of Masoliver {\it et al}., 1993, (20) of Kolesnik and Orsingher, 2005, and for the case $d=4,h=2,$ see formula (3.2) of Orsingher and De Gregorio, 2007).

Therefore, if we assume that the changes of direction are governed by the homogeneous Poisson process $\{N(t),t>0\}$, the related reflecting random flights $\{\underline{\bf X}_d^*(t),t>0\},d=2,4,$ have unconditional density functions (in a generalized sense) equal to
\begin{align}
f_{h}^*(\underline{{\bf x}}_d,t)&=p_h(\underline{{\bf x}}_d,t){\bf 1}_{B_R^d}(\underline{{\bf x}}_d)+\frac{R^{2d}}{||\underline{{\bf x}}_d||^{2d}}p_h\left(R^2\frac{\underline{{\bf x}}_d}{||\underline{{\bf x}}_d||^2},t\right) {\bf 1}_{C^d_{\frac{R^2}{ct},R}}(\underline{{\bf x}}_d)\notag\\
&\quad+\frac{e^{-\lambda t}}{\text{area}\left(\se_{R^2/ct}^{d-1}\right)}\delta_{\{R^2/ct\}}(||\underline{{\bf x}}_d||),
\end{align}
where $p_h(\underline{{\bf x}}_d,t)$ is given by \eqref{eq:dgo1} and the $\frac{e^{-\lambda t}}{\text{area}\left(\se_{R^2/ct}^{d-1}\right)}\delta_{\{R^2/ct\}}(||\underline{{\bf x}}_d||)$ emerges if $N(t)=0$.

\begin{theorem}
For $t>\frac Rc$, $0<r\leq R$ and $d=2, h=1,$ we obtain that 
\begin{align}
P\{D_2^*(t)<r\}&=\left[1-\exp\left\{-\lambda t+\frac{\lambda}{c}\sqrt{c^2t^2-r^2}\right\}\right]{\bf 1}_{(0,R]}(r)\notag\\
&\quad+\exp\left\{-\lambda t+\frac{\lambda}{c}\sqrt{c^2t^2-R^4/r^2}\right\}{\bf 1}_{(R^2/ct,R]}(r)
\end{align}
\end{theorem}
\begin{proof}
Let us fix $r\in(0,R]$ and $d=2, h=1$. We have that 
\begin{align*}
P\{D_2^*(t)<r\}&= P\{\underline{{\bf X}}_2^*(t)\in B_r^2\}\\
&=\int_{  B_r^2}p_1(\underline{\bf x}_2,t)\de x_1\de x_2+\int_{ C^2_{\frac{R^2}{ct},R}\cap B_r^2}\frac{R^{4}}{||\underline{{\bf x}}_2||^{4}}p_1\left(\frac{R^2}{||\underline{{\bf x}}_2||},t\right)\de x_1 	\de x_2.
\end{align*}
where $p_1\left(\underline{\bf x}_2,t\right)$ is given by \eqref{eq:dgo1}.
 If $r<R^2/ct$, it is clear that $C^2_{\frac{R^2}{ct},R}\cap B_r^2=\varnothing$ and then
\begin{align*}
P\{\underline{{\bf X}}_2^*(t)\in B_r^2\}
&=\int_{  B_r^2}\frac{\lambda e^{-\lambda t}}{2\pi c}\frac{e^{\frac{\lambda}{c}\sqrt{c^2t^2-||\underline{\bf x}_2||^2}}}{\sqrt{c^2t^2-||\underline{\bf x}_2||^2}}\de x_1\de x_2\\
&=\int_0^r\frac{\lambda e^{-\lambda t}}{c}\rho\frac{e^{\frac{\lambda}{c}\sqrt{c^2t^2-\rho^2}}}{\sqrt{c^2t^2-\rho^2}}\de \rho\\
&=1-\exp\left\{-\lambda t+\frac{\lambda}{c}\sqrt{c^2t^2-r^2}\right\}
\end{align*}

For $r\in(R^2/ct,R]$, we have that $C^2_{\frac{R^2}{ct},R}\cap B_r^2=C^2_{\frac{R^2}{ct},r}$. Therefore
\begin{align*}
\int_{ C^2_{\frac{R^2}{ct},r}}\frac{\lambda e^{-\lambda t}}{2\pi c}\frac{R^{4}}{||\underline{{\bf x}}_2||^4}\frac{e^{\frac{\lambda}{c}\sqrt{c^2t^2-R^4/||\underline{{\bf x}}_2||^2}}}{\sqrt{c^2t^2-R^4/||\underline{{\bf x}}_2||^2}}\de x_1 \de x_2&=\int_{R^2/ct}^r\frac{\lambda e^{-\lambda t}}{c}\frac{R^{4}}{\rho^3}\frac{e^{\frac{\lambda}{c}\sqrt{c^2t^2-R^4/\rho^2}}}{\sqrt{c^2t^2-R^4/\rho^2}}\de \rho\\
&=\exp\left\{-\lambda t+\frac{\lambda}{c}\sqrt{c^2t^2-R^4/r^2}\right\}-\exp\{-\lambda t\}.
\end{align*}
If $R^2/ct<r\leq R$, we also have to consider the discrete part of the distribution of $\{\underline{{\bf X}}_2^*(t),t>0\}.$ Therefore
$$P\{D_2^*(t)=R^2/ct\}=P\{\underline{{\bf X}}_2^*(t)\in \se_{R^2/ct}^{d-1}\}=P\{N(t)=0\}=e^{-\lambda t}.$$
This last fact concludes the proof of the theorem.
\end{proof}

\section{Reflecting random flights on hyperplanes}

\subsection{Definitions and probability distributions}

In this section we introduce a random flight bouncing off a hyperplane. 
Let $H(\underline a_d,b):=\{\underline{{\bf x}}_d\in\re ^d:\, \langle\underline a_d,\underline{{\bf x}}_d\rangle=b;\,\underline a_d\in\re^d,b\in\re\}$ be a hyperplane in $\re^d$. A random flight starting from the origin of $\re^d,$ for sufficiently large values of $t$ can be located beyond the hyperplane $H(\underline a_d,b)$. The spherical set of the possible positions $B_{ct}^d$ is therefore composed by the set $L_{ct}^d:=L_{ct}^d(\underline a_d, b):=\{\underline{{\bf x}}_d\in \mathbb{R}^d:||\underline{{\bf x}}_d||^2< c^2t^2,\langle\underline a_d,\underline{{\bf x}}_d\rangle<b\}$ pertaining to the sample paths which have not crossed $H(\underline a_d,b)$ and the set $U_{ct}^d:=U_{ct}^d(\underline a_d, b):=\{\underline{{\bf x}}_d\in \mathbb{R}^d:||\underline{{\bf x}}_d||^2< c^2t^2,\langle\underline a_d,\underline{{\bf x}}_d\rangle\geq b\}$ related to the trajectories which have gone beyond the hyperplane. Of course, if no deviation is recorded by the random flight up to time $t$, the moving particle attains the sphere $\se_{ct}^{d-1},$ which therefore can be split as $\se_{ct}^{d-1}=\partial L_{ct}^d\cup \partial U_{ct}^d,$ where $\partial L_{ct}^d:=\partial L_{ct}^d(\underline a_d, b):=\{\underline{{\bf x}}_d\in \mathbb{R}^d:||\underline{{\bf x}}_d||^2= c^2t^2,\langle\underline a_d,\underline{{\bf x}}_d\rangle<b\}$ and $\partial U_{ct}^d:=\partial U_{ct}^d(\underline a_d, b):=\{\underline{{\bf x}}_d\in \mathbb{R}^d:||\underline{{\bf x}}_d||^2=c^2t^2,\langle\underline a_d,\underline{{\bf x}}_d\rangle\geq b\}$.

The reflection of the sample paths crossing $H(\underline a_d,b)$ is described in detail in Appendix \ref{refhyper}. Substantially, the incoming and reflected sample paths form the same angle $\theta$ w.r.t. the normal to the hyperplane. Let $\nu:\re^d\to\re$ be the reflecting (bijective) operator with respect to the hyperplane  $H(\underline a_d,b)$ defined as

\begin{equation*}
\nu(\underline{{\bf x}}_d):=\underline{{\bf x}}_d+2\frac{ b-\langle\underline a_d,\underline{{\bf x}}_d\rangle}{\langle\underline a_d,\underline a_d\rangle}\underline a_d.
\end{equation*}
Now, we are able to define the reflecting random flight.  Let $t':=\inf(t:H(\underline a_d,b)\cap  B_{ct}^d\neq \varnothing )$.

\begin{definition}\label{defrefhyper}
The reflecting random flight $\{\underline{{\bf X}}_d'(t),t>0\},$ reflected by the hyperplane $H(\underline a_d,b)$ is constructed by means of the free process $\{\underline{{\bf X}}_d(t),t>0\}$ as follows: 1) if $t<t'$, then $\underline{{\bf X}}_d'(t)=\underline{{\bf X}}_d(t);$ 2) if $t\geq t'$ and at least one change of direction happens during the time interval $[0,t]$, we have that
\begin{align}
\underline{{\bf X}}_d'(t)&=\underline{{\bf X}}_d(t){\bf 1}_{L_{ct}^d }(\underline{{\bf X}}_d(t))+\nu(\underline{{\bf X}}_d(t)){\bf 1}_{U_{ct}^d}(\underline{{\bf X}}_d(t)),
\end{align}
while if no deviation up to time $t$ is recorded
\begin{align}
\underline{{\bf X}}_d'(t)&=\underline{{\bf X}}_d(t){\bf 1}_{\partial L_{ct}^d }(\underline{{\bf X}}_d(t))+\nu(\underline{{\bf X}}_d(t)){\bf 1}_{\partial U_{ct}^d}(\underline{{\bf X}}_d(t)).
\end{align}
\end{definition}

\begin{figure}[t]
\begin{center}
\includegraphics[angle=0,width=0.6\textwidth]{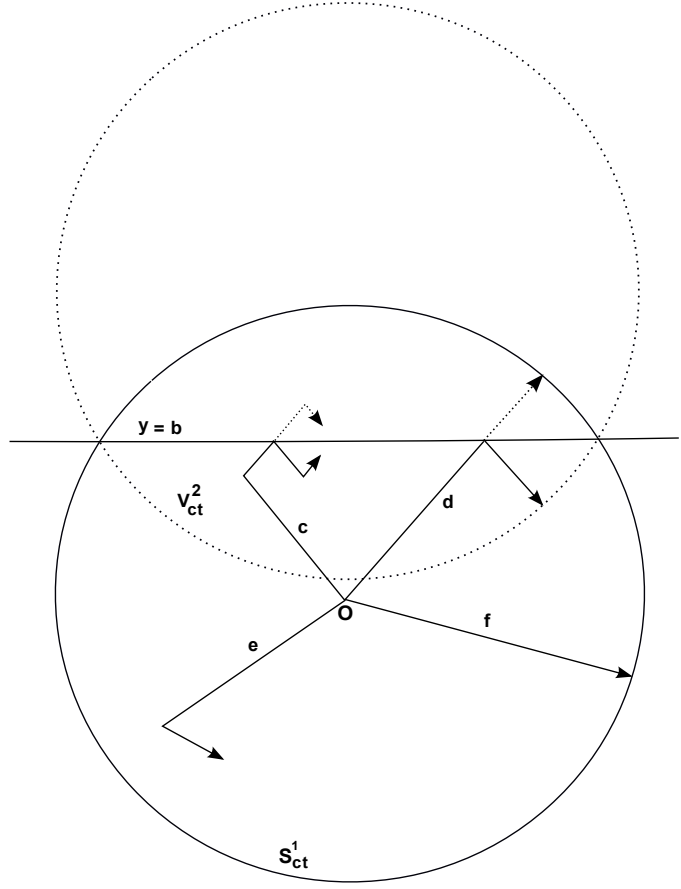}
\caption{Four typical sample paths are depicted. The trajectories ${\bf c}$ and ${\bf d}$ are reflected in the set $V_{ct}^d$, while ${\bf e}$ and ${\bf f}$ never cross the reflecting surface $y=b$.}\label{fig2}
\end{center}
\end{figure}

The set of points of $B_{ct}^d$ obtained by reflection $\nu$ around $H(\underline a_d,b)$ is denoted by $V_{ct}^d$, that is $V_{ct}^d:=V_{ct}^d(\underline a_d, b):=\{\underline{{\bf x}}_d\in \mathbb{R}^d:||\nu(\underline{{\bf x}}_d)||^2< c^2t^2,\langle\underline a_d,\underline{{\bf x}}_d\rangle\leq b\}$, while $\partial V_{ct}^d$ stands for the set obtained by the reflection of $\partial U_{ct}^d$ (see Figure \ref{fig2}). The reflection at the hyperplane preserves the form of the sample paths of the free random flight; i.e. the sample paths of $\{\underline{{\bf X}}_d'(t),t>0\}$ are straight lines as well as the trajectories of the free random flights (see Figure \ref{fig2}). Furthermore, the property Q4) guarantees that the reflected trajectories are symmetrically specular w.r.t. the hyperplane $H(\underline a_d,b)$.

 The conditional distributions of the reflecting random flight $\{\underline{{\bf X}}_d'(t),t>0\}$ are given in the next theorem.
\begin{theorem}
If  $\mathcal N(t)=n$, with $n\geq 1,$ the process $\{\underline{{\bf X}}_d'(t),t>0\},$ has the following conditional density functions
\begin{align}\label{eq:condens2}
p_{n}'(\underline{{\bf x}}_d,t)=\begin{cases}
p_n(\underline{{\bf x}}_d,t){\bf 1}_{B_{ct}^d}(\underline{{\bf x}}_d),& t\leq t'\\
p_n(\underline{{\bf x}}_d,t){\bf 1}_{L_{ct}^d}(\underline{{\bf x}}_d)+p_n\left(\nu(\underline{{\bf x}}_d),t\right){\bf 1}_{V_{ct}^d}(\underline{{\bf x}}_d),& t>t'.
\end{cases}
\end{align}
where $p_n(\underline{{\bf x}}_d,t)$ is equal to \eqref{eq:densfree}.
\end{theorem}
\begin{proof} 
The case $t\leq t'$ is trivial. We assume that $t>t'.$
Let $A$ be  a Borel set such that $A\cap H(\underline a_d,b)\neq \varnothing$. We observe that
\begin{align}\label{eq:step1dens2}
P_n\{\underline{{\bf X}}_d'(t)\in A\}&=P_n\{\underline{{\bf X}}_d(t)\in A\cap L_{ct}^d\}+P_n\{\nu(\underline{{\bf X}}_d(t))\in A\cap V_{ct}^d\}\notag\\
&=\int_{ A\cap L_{ct}^d}p_n(\underline{{\bf x}}_d,t)\prod_{k=1}^d\de x_k+P_n\{\nu(\underline{{\bf X}}_d(t))\in A\cap V_{ct}^d\}.
\end{align} 
Let now
$$B_A:=\left\{\underline{{\bf y}}_d\in \re^d: \underline{{\bf x}}_d=\nu(\underline{{\bf y}}_d)\in A\cap V_{ct}^d\right\}$$
and thus, by means of Jacobi's transformation formula, we have that
\begin{align}\label{eq:step2dens2}
P_n\{\nu(\underline{{\bf X}}_d(t))\in A\cap V_{ct}^d\}&=P_n\{\underline{{\bf X}}_d(t)\in B_A\}\notag\\
&=\int_{B_A}p_n(\underline{{\bf y}}_d,t)\prod_{k=1}^d\de y_k\notag\\
&=\int_{A\cap V_{ct}^d} p_n\left(\nu^{-1}(\underline{{\bf x}}_d),t\right)|\text{det}(J_{\nu^{-1}}(\underline{{\bf x}}_d))|\prod_{k=1}^d\de x_k\notag\\
&=\int_{A\cap V_{ct}^d} p_n\left(\nu(\underline{{\bf x}}_d),t\right)\prod_{k=1}^d\de x_k
\end{align}
where in the last step we have exploited the facts: $\nu=\nu^{-1}$ (which follows from the property Q3)) and $|\text{det}(J_{\nu}(\underline{{\bf x}}_d))|=1$. From \eqref{eq:step1dens2} and \eqref{eq:step2dens2} the result \eqref{eq:condens2} immediately follows.
\end{proof}

\begin{remark}
In view of the property \eqref{eq:normrefl2}, the reflecting random flights introduced by Definition \ref{defrefhyper} are no longer isotropic. Indeed, for $t>t'$, in the density function \eqref{eq:condens2} appears 
$$p_{n}(\nu(\underline{{\bf x}}_d),t)=p_{n}(||\nu(\underline{{\bf x}}_d)||,t)$$
which does not depend only on the Euclidean distance $||\underline{{\bf x}}_d||.$
\end{remark}

Now, we consider reflecting Dirichlet random flights.
The assumption \eqref{eq:dirichlet} implies that the reflecting process  $\{\underline{{\bf X}}_d'(t),t>0\}$ has probability law \eqref{eq:condens2} given by 
\begin{align}\label{eq:densrefl2}
p_{n,h}'(\underline{{\bf x}}_d,t)&=\frac{\Gamma(\frac{n+1}{2}(d-h)+\frac h2)}{\Gamma(\frac n2(d-h))}\frac{1}{\pi^{\frac d2}(ct)^{(n+1)(d-h)+h-2}}\\
&\quad\times\left[(c^2t^2-||\underline{{\bf x}}_d||^2)^{\frac n2(d-h)-1}{\bf 1}_{L_{ct}^d}(\underline{{\bf x}}_d)+\left(c^2t^2-||\nu(\underline{{\bf x}}_d)||^2\right)^{\frac n2(d-h)-1}{\bf 1}_{V_{ct}^d}(\underline{{\bf x}}_d)\right],\notag
\end{align}
where $t>t'$ and $||\nu(\underline{{\bf x}}_d)||^2$ is given by \eqref{eq:normrefl2}. In the special cases (i)-(iv) mentioned in Remark \ref{remarkun}, the function \eqref{eq:densrefl2} becomes 
\begin{align}
p_{n,h}'(\underline{{\bf x}}_d,t)&=\frac{1}{\text{area}( \se_{ct}^{d-1})}\left[{\bf 1}_{L_{ct}^d}(\underline{{\bf x}}_d)+{\bf 1}_{V_{ct}^d}(\underline{{\bf x}}_d)\right],\quad t>t'.
\end{align}

By assuming that the random number of changes of direction has probability law \eqref{lawgenpoi}, for $t>t',$ we have that the unconditional probability distribution of $\{\underline{\bf X}_{d}'(t),t>0\}$ becomes
\begin{equation}
P\{\underline{\bf X}_{d}'(t)\in \de\underline{\bf x}_{d}\}=p_h'(\underline{\bf x}_d,t)\prod_{k=1}^d\de x_k+\frac{1}{E_{\frac{d-h}{2},\frac d2}(\lambda t)\Gamma(\frac d2)}\mu'(\de\underline{\bf x}_d )
\end{equation}
where
\begin{align}\label{eq:unconddensrefl}
p_h'(\underline{\bf x}_d,t)&=\frac{1}{(ct)^d\pi^{\frac d2}}\frac{1}{E_{\frac{d-h}{2},\frac d2}(\lambda t)}\left\{\left[\gamma_h(||\underline{\bf x}_d||,t)\right]^{1-\frac{2}{d-h}}E_{\frac {d-h}{2},\frac {d-h}{2}}\left(\gamma_h(||\underline{\bf x}_d||,t)\right){\bf 1}_{L_{ct}^d}(\underline{{\bf x}}_d)\right.\notag\\
&\quad+\left.\left[\gamma_h\left(||\nu(\underline{\bf x}_d)||,t\right)\right]^{1-\frac{2}{d-h}}E_{\frac {d-h}{2},\frac {d-h}{2}}\left(\gamma_h\left(||\nu(\underline{\bf x}_d)||,t\right)\right){\bf 1}_{V_{ct}^d}(\underline{{\bf x}}_d)\right\},
\end{align}
with 
$$\gamma_h(||\underline{\bf x}_d||,t):=\lambda t\left(1-\frac{||\underline{\bf x}_{d}||^2}{c^2t^2}\right)^{\frac {d-h}{2}},$$
and $\mu'$ represents the uniform law on $\partial L_{ct}^d\cup \partial V_{ct}^d.$

\begin{remark}
It is not hard to prove that the function
$$\hat f_\beta(\underline{\bf x}_d,t):=\left(c^2t^2-||\nu(\underline{{\bf x}}_d)||^2\right)^\beta,\quad \underline{\bf x}_d\in V_{ct}^d,\beta\in\re,$$
is a solution of the partial differential equation \eqref{eq:epdcon}. Therefore, the second component appearing in \eqref{eq:unconddensrefl}, that is
$$\hat p_{n,h}(\underline{{\bf x}}_d,t):=\frac{\Gamma(\frac{n+1}{2}(d-h)+\frac h2)}{\Gamma(\frac n2(d-h))}\frac{\left(c^2t^2-||\nu(\underline{{\bf x}}_d)||^2\right)^{\frac n2(d-h)-1}}{\pi^{\frac d2}(ct)^{(n+1)(d-h)+h-2}},\quad \quad \underline{\bf x}_d\in V_{ct},$$
is a solution of the EPD equation \eqref{eq:epdcon2}.
\end{remark}

\subsection{On the probability law of the distance from the origin}
For the sake of simplicity, we set $\underline a_d=e_d:=(0,...,0,1)$. Therefore, the hyperplane becomes
$$H(e_d,b)=\{\underline{{\bf x}}_d\in\re ^d:\, x_d=b;\,b>0\}$$
and the reflection map becomes
$$\nu(\underline{{\bf x}}_d):=\underline{{\bf x}}_d+2 (b-x_d),$$with $$||\nu(\underline{{\bf x}}_d)||^2=||\underline{{\bf x}}_d||^2+4b^2-4bx_d.
$$
Under the above assumption, we have that
\begin{align*}
&L_{ct}^d=\{\underline{{\bf x}}_d\in \mathbb{R}^d:||\underline{{\bf x}}_d||^2< c^2t^2,x_d<b\},\\
&V_{ct}^d=\{\underline{{\bf x}}_d\in \mathbb{R}^d:||\underline{{\bf x}}_d||^2+4b^2-4bx_d< c^2t^2,x_d\leq b\}.
\end{align*}

Let $\{D_d'(t),t>0\}$ where $D_d'(t)=||\underline{{\bf X}}_d'(t)||$. Since the process $\{\underline{{\bf X}}_d'(t),t>0\}$ is not isotropic, the probability distribution of $\{D_d'(t),t>0\}$ is more complicated than \eqref{eq:distancecircinv}. Now, we consider the distribution function
$$P_n\{D_d'(t)<r\}$$
with $0<r<ct.$ For $t>t',$ we distinguish the following three cases (see Figure \ref{fig3}):
\begin{itemize}
\item[1.] $0<r<ct-2b$, where the ball $B_r^d$ does not intersect $V_{ct}^d$;
\item[2.] $ct-2b<r<b$, where $B_r^d$ intersects $V_{ct}^d$ but does not overlap $H(e_d,b)$;
\item[3.] $b<r<ct$, where $B_r^d$ intersects $V_{ct}^d$ and $H(e_d,b)$.
\end{itemize}

\begin{figure}[t]
\begin{center}
\includegraphics[angle=0,width=0.6\textwidth]{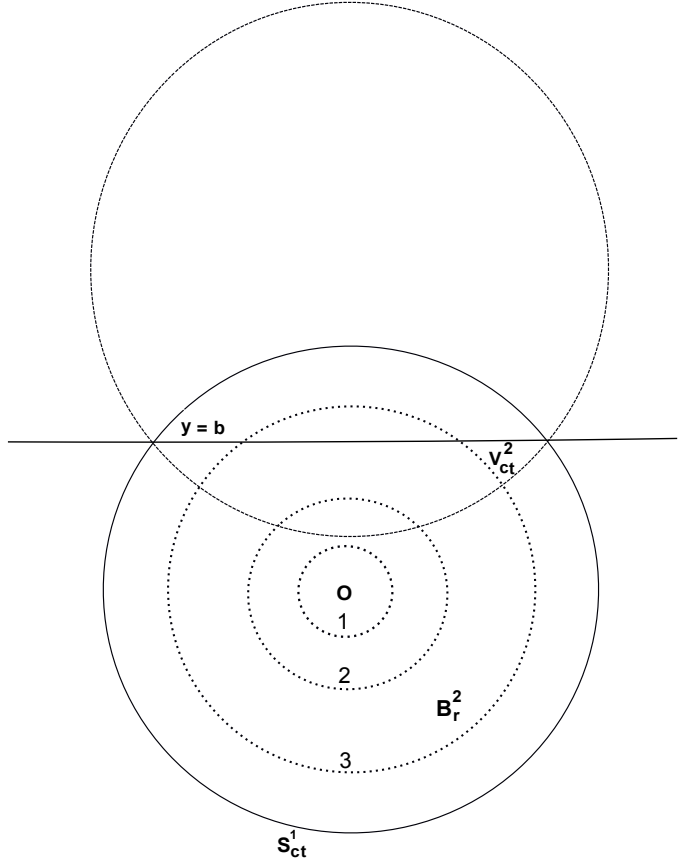}
\caption{The picture represents cases 1., 2. and 3. emerging in the analysis of $P_n\{D_d'(t)<r\}.$}\label{fig3}
\end{center}
\end{figure}

In the case (i), we simply have that
\begin{align}
P_n\{D_d'(t)<r\}=P_n\{D_d(t)<r\}=P_n\{\underline{{\bf X}}_d(t)\in B_{r}^d\}.
\end{align}
In the second case, we must take into account that in $V_{ct}^d\cap  B_{r}^d$ we meet reflected sample paths and thus
\begin{align}
P_n\{D_d'(t)<r\}=P_n\{\underline{{\bf X}}_d(t)\in B_{r}^d\}+P_n\{\nu(\underline{{\bf X}}_d(t))\in V_{ct}^d\cap  B_{r}^d\}.
\end{align}
In the third case $B_r^d=L_r^d\cup U_r^d$ and thus
\begin{align}\label{eq:distancethird}
P_n\{D_d'(t)<r\}=P_n\{\underline{{\bf X}}_d(t)\in L_{r}\}+P_n\{\nu(\underline{{\bf X}}_d(t))\in V_{ct}^d\cap  L_{r}\}.
\end{align}
The reader should consider that sample paths crossing $x_d=b$ and outside $U_r^d$ can contribute to probability \eqref{eq:distancethird} because the reflected trajectories lie within $V_{ct}^d\cap  L_{r}$. All these considerations can be summarized as follows
\begin{align}
P_n\{D_d'(t)<r\}=
\begin{cases}
\int_{B_{r}^d}p_{n}(\underline{{\bf x}}_d,t)\prod_{k=1}^d\de x_k,& 0<r<ct-2b,\\
\int_{B_{r}^d}p_{n}(\underline{{\bf x}}_d,t)\prod_{k=1}^d\de x_k+\int_{ V_{ct}^d\cap  B_{r}^d}p_{n}(\nu(\underline{{\bf x}}_d),t)\prod_{k=1}^d\de x_k,& ct-2b<r<b,\\
\int_{L_{r}^d}p_{n}(\underline{{\bf x}}_d,t)\prod_{k=1}^d\de x_k+\int_{ V_{ct}^d\cap  L_{r}^d}p_{n}(\nu(\underline{{\bf x}}_d),t)\prod_{k=1}^d\de x_k,& b<r<ct.
\end{cases}
\end{align}

 \appendix
\section{Reflection in spheres}\label{refspheres}

We recall the basic facts about the circular inversion or reflection of a point inside a sphere (see, for instance, Ratcliffe, 2006, and Wong, 2009). If we consider a point $\underline{{\bf x}}_d$ inside $\se_R^{d-1}(\underline{{\bf x}}_d^0)$, having polar coordinates equal to $(r,\underline{\theta}_{d-1})$, we can find another point $\underline{{\bf x}}'_d$ in the space $\mathbb{R}^d$ with polar coordinates given by $(r',\underline{\theta}_{d-1})$ ($R<r'$ and the same angle $,\underline{\theta}_{d-1}$), such that $$rr'=R^2.$$
or equivalently 
$$||\underline{{\bf x}}_d-\underline{{\bf x}}_d^0||\cdot ||\underline{{\bf x}}'_d-\underline{{\bf x}}_d^0||=R^2.$$
The point $\underline{{\bf x}}_d$ is called the inverse point of $\underline{{\bf x}}'_d$ with respect to $\se_R^{d-1}(\underline{{\bf x}}_d^0)$. The circular inversion in $\se_R^{d-1}(\underline{{\bf x}}_d^0)$ is defined as the bijective map
$\mu_{R,\underline{{\bf x}}_d^0}:\re^d\setminus \{O\}\to \re^d\setminus \{O\}$ defined as follows
\begin{align}\label{eq:appendixmap1}
\mu_{R,\underline{{\bf x}}_d^0}(\underline{{\bf x}}_d)= R^2\frac{\underline{{\bf x}}_d-\underline{{\bf x}}_d^0}{||\underline{{\bf x}}_d-\underline{{\bf x}}_d^0||^2}+\underline{{\bf x}}_0^d.
\end{align}
We set $\mu_{R,O}(\underline{{\bf x}}_d):=\mu_{R}(\underline{{\bf x}}_d)=R^2\frac{\underline{{\bf x}}_d}{||\underline{{\bf x}}_d||^2}$. 
The map $\mu_R$ has the following properties:
\begin{itemize}
\item[P1)] the points inside the sphere are taken to points outside it and vice versa;
\item[P2)] $\mu_R(\underline{{\bf x}}_d)=\underline{{\bf x}}_d$ if and only if $\underline{{\bf x}}_d\in \se_R^{d-1}$;
\item[P3)] $(\mu_R\circ \mu_R)(\underline{{\bf x}}_d)=\underline{{\bf x}}_d$ for all $\re^d\setminus \{O\}$;
\item[P4)] the map $\mu_R$ is conformal and reverses orientation (that is det$(\mu'_R(\underline{{\bf x}}_d))<0$);
\item[P5)] the inversion $\mu_R$ maps straight lines into a straight line or sphere. In other words, lines passing through the center of inversion are mapped into themselves; while lines not passing through the center of inversion are mapped into spheres passing through the center.
\end{itemize}

\section{Reflection in hyperplanes}\label{refhyper}
For the main aspects on the reflection in hyperplanes consult Ratcliffe (2006). Let us consider a hyperplane of $\re^d$ given by
\begin{equation}\label{eq:hyperplane}
H(\underline a_d,b)=\{\underline{{\bf x}}_d\in\re ^d:\, \langle\underline a_d,\underline{{\bf x}}_d\rangle=b;\,\underline a_d\in\re^d,b\in\re\}.
\end{equation}

Let $\nu:\re^d\to \re^d$ be the reflection map in the hyperplane $H(\underline a_d,b)$ which is a bijection defined as
\begin{equation}\label{eq:reflmaphyper}
\nu(\underline{{\bf x}}_d):=\underline{{\bf x}}_d+2\frac{ b-\langle\underline a_d,\underline{{\bf x}}_d\rangle}{\langle\underline a_d,\underline a_d\rangle}\underline a_d,
\end{equation}
with

\begin{equation}\label{eq:normrefl2}
||\nu(\underline{{\bf x}}_d)||^2=||\underline{{\bf x}}_d||^2+\frac{ 4b^2-4b\langle\underline a_d,\underline{{\bf x}}_d\rangle}{\langle\underline a_d,\underline a_d\rangle}
\end{equation}
The map $\nu(\underline{{\bf x}}_d)$ is defined as the mirror image of $\underline{{\bf x}}_d$ across $H(\underline a_d,b)$. Furthermore, $\nu$ has the following properties (see, for instance, Ratcliffe, 2006):
\begin{itemize}
\item[Q1)] the points inside the hyperplane are taken to points outside it and vice versa;
\item[Q2)] $\nu(\underline{{\bf x}}_d)=\underline{{\bf x}}_d$ if and only if $\underline{{\bf x}}_d\in H(\underline a_d,b)$;
\item[Q3)] $(\nu\circ \nu)(\underline{{\bf x}}_d)=\underline{{\bf x}}_d$ for all $\underline{{\bf x}}_d\in\re^d$;
\item[Q4)] the map $\nu$ is conformal and reverses orientation (that is det$(\nu'(\underline{{\bf x}}_d))<0$);
\item[Q5)] $\nu$ is an isometry.
\end{itemize}

\end{document}